\documentclass[english]{amsart}
\usepackage{amsmath,amsthm,amssymb,babel,amsfonts,graphics}

\newtheorem{thm}{Theorem}[]
\newtheorem{cor}{Corollary}
\newtheorem{lem}{Lemma}

\newtheorem{defn}{Definition}

\newcommand{\eps}{\text{$\varepsilon$}}
\newcommand{\lb}{\left\{\right.}
\newcommand{\rb}{\left.\right\}}

\newcommand{\IH}{\text{${\mathbb{H}}$}}
\newcommand{\IR}{\text{${\mathbb{R}}$}}
\newcommand{\IRd}{\text{${\mathbb{R}^d}$}}
\newcommand{\IC}{\text{${\mathbb{C}}$}}

\newcommand{\IZ}{\text{${\mathbb{Z}}$}}
\newcommand{\IZd}{\text{${\mathbb{Z}^d}$}}

\newcommand{\IN}{\text{${\mathbb{N}}$}}
\newcommand{\IT}{\text{${\mathbb{T}}$}}

\newcommand{\x}{\bf {x}}

\renewcommand{\t}{\mathbf{t}}
\renewcommand{\k}{\mathbf{k}}

\newcommand{\comment}[1]{}

\newcounter{ang}
\newcounter{zal}
\title{On Parseval wavelet frames via   Multiresolution Analyses}

\author{A. San Antol\'{\i}n}

\address{Departamento de  Matem\'aticas,  Universidad de Alicante,
 03080 Alicante, Spain.}\email{angel.sanantolin@ua.es}

\thanks{The author was
partially supported  by  MEC/MICINN grant  \#MTM2011-27998 (Spain)  and by Generalitat Valenciana  grant GV/2015/035}


\date{}

\begin{document}

\begin{abstract}
We give a characterization of all  Parseval  wavelet frames  arising from a given frame multiresolution analysis. As a consequence, we obtain a description of all  Parseval  wavelet frames  associated with a frame multiresolution analysis. These results are based on a version of Unitary Extension Principle and Oblique Extension Principle with the  assumption that the origin is a point of approximate continuity of the Fourier transform of the involved refinable functions.
\end{abstract}

\maketitle

\pagestyle{myheadings}

\markboth{ON PARSEVAL WAVELET FRAMES VIA  MULTIRESOLUTION ANALYSES}{SAN ANTOL\'IN}


\section{Introduction} \label{section1}

We are interested in the study of methods for constructing tight wavelet frames.
In this paper,  we emphasis on tight wavelet frames can be constructed via multiresolution analyses and  extension principles. Mallat \cite{M:89} and Meyer
\cite{Me:90} introduced the definition of multiresolution analysis (MRA) as a general method for constructing wavelets. In order to construct wavelet frames, the requirements on the definition of MRA were weakened. In this sense,  the notion of a frame
multiresolution analysis (FMRA)   was formulated by
Benedetto and  Li \cite{BL:98} as a natural
extension of MRA. 
Furthermore, a generalized multiresolution analysis was first   introduced  by Baggett,  Medina and Merril \cite{BMM:99} and  Papadakis \cite{P:01} independently, see also the paper by
de Boor,  DeVore and A. Ron  \cite{BDR:93}.
 Gripenberg \cite{Gr:95} and Hern\'andez and Weiss \cite{HW:96} proved a characterization of all orthonormal wavelets associated to a multiresolution analysis in terms of its wavelet dimension function. Any orthonormal wavelet is associated with a generalized multiresolution analysis was proved by Papadakis \cite{Pa:99}. In the paper by Kim, Kim and Lim \cite{KKL:01} (see also Kim, Kim, Lee and Lim \cite{KKLL:02}),   characterizations of the Riesz wavelets which are associated with a multiresolution analysis were proved. A generalization of theses results are given by Bownik and Garrig\'os \cite{BG:04}. Note that characterizations of biorthogonal wavelets from biorthogonal multiresolution analysis are proved in \cite{KKL:01}, \cite{BG:04} (see also Calogero and Garrig\'os \cite{CG:01}).  Zalik \cite{Za:99}
 introduced the notion of Riesz wavelet obtained by a multiresolution analysis. Moreover, he gave  necessary and sufficient conditions on a given Riesz wavelet to be obtained by a multiresolution analysis. Bownik \cite{B:03} studied both the notion of Riesz basis associated with a generalized  multiresolution analysis and Riesz basis obtained by a generalized multiresolution analysis and proved that these two notions are equivalent. Charazterizations of Parseval wavelet frames associated with generalized multiresolution analysis are proved by Baggett, Medina and Merrill \cite{BMM:99} and by Baki\'c \cite{B:06}.

 A slight different point of view for  constructing  wavelet frames was first proposed in the Unitary Extension  Principle (UEP) by  Ron and Shen \cite{RS1} (see also \cite{RS11}). The UEP leads to explicit constructions of tight wavelet frames generated by a refinable function.
A more flexible way for constructing wavelet frames is the so called Oblique Extension Principle (OEP). The OEP was introduced by Daubechies, Han, Ron and Shen \cite{DHRS}. These extension principles have been developed by Benedetto and  Treiber \cite{BT:01},  Chui, He and St\"ockler \cite{CHS:02},  Chui,  He,  St\"ockler and Sun \cite{CHSS:03}, Han   \cite{H:10}, \cite{H:12}, Stavropoulos \cite{S:12} and
 Atreas, Melas and Stavropoulos \cite{AMS:14}. Observe that in these papers, there are also proved that the obtained
  sufficient conditions  are also necessary.

  Extensive studies on multiresolution analysis and extension principles are enclosed, for instance, in \cite{Ch:03}.

Here, we give a solution to the problem of characterizing all Parseval wavelet frames arising from a fix frame multiresolution analysis. As a consequence, we obtain a new description of all Parseval wavelet frames associated with  a frame multiresolution analysis. The proofs of these results are based on the characterization of the scaling functions in \cite{KS:08} and on a version of Unitary Extension Principle and Oblique Extension Principle with no regularity at the origin on the modulus of the Fourier transform of the involved refinable functions. In particular, we invoke the notion of approximate continuity. Furthermore, we characterize all Parseval wavelet frames  can be constructed via Oblique Extension Principles. We observe that the origin  must be a point of density for the support of the Fourier transform of a refinable function used for constructing Parseval wavelet frames via OEP.

Versions of the UEP and OEP without any extra condition on refinable functions where first proved by  Han   \cite{H:10}, \cite{H:12}, were the contexts is in distribution spaces. The assumption at the origin of the Fourier transform of involved refinable functions is a limit in sense of distributions. At the end of this manuscript we will see  that  the condition used by Han and the condition used here are of different nature.

Although the results we present here are new in the classical context, i.e., in $L^2(\IR)$ with the dyadic dilation, we consider functions in $L^2(\IR^n)$, $n \geq 1$, and the dilation is given by $A: \IR^n \to \IR^n$, a fix expansive linear map such that $A(\IZ^n) \subset \IZ^n$.

Before formulating our results let us introduce some notation and definitions.

The sets of strictly positive integers, integers, real and complex numbers will be
denoted by $\IN$, $\IZ$ ,  $\IR$ and $\IC$ respectively.
 $\IT^n =
~\IR^n/\IZ^{n}$, $ n \geq 1$,  and with  some abuse of the notation we
consider
 also that $\IT^n$ is the unit cube
 $[0,1)^{n}.$

   We will denote $B_r({\bf y}) = \{ {\bf
 x}\in \IR^n : |{\bf x} - {\bf y}| <r\},$ and   will write $B_r$ if $ {\bf y}$ is the
 origin. If $A:\IR^n \to \IR^n$ is a linear map, $A^*$ will mean the adjoint of $A$.
 With some abuse in the notation, if we write $A$ we also mean the corresponding matrix respect to the canonical basis.
 Moreover $d_A= | \det A |$.
  For a Lebesgue measurable set $E\subset\IR^n$, $E^c= \IR^n \setminus E$ and the Lebesgue
measure of $E$ in $\IR^n$ will be denoted by $|E|_n$. If ${\bf
 x}\in \IR^n$ then ${\bf
 x}+E = \{ {\bf
 x}+ {\bf y} : \quad\mbox{for}\quad {\bf y} \in E \}.$  We will denote $ A(E) = \lb {\bf x}\in \IR^n: {\bf
 x}= A({\bf t}) \quad\mbox{for}\quad {\bf t}\in E\rb$ and the volume of $E$ changes under $A$ according to $|AE|_n= d_A |E|_n$.
 The characteristic function of a set $E \subset \IR^n$ will be denoted by $\chi_E$, i.e. $\chi_E({\x})=1$ if ${\x} \in E$ and
 $\chi_E({\x})=0$ otherwise.

  If we write  $L^p(\IR^n)$, $n \geq 1$, $1 \leq p \leq \infty$,
we mean the usual Lebesgue space. If we write $f \in L^p_{loc}(\IR^n) $ we mean the linear space of
all measurable functions such that $f \chi_K \in L^p(\IR^n)$ for every $K \subset  \IR^n$ compact sets.

If we
 take $f\in L^p(\IT^n)$ we will understand that $f$ is defined on
 the whole space $\IR^n$ as a $\IZ^n$-periodic function.

 A linear map  $A$ is called expansive if all
(complex) eigenvalues of $A$ have absolute value greater than $1$.
If $A$ is invertible, we consider the  operator $D_A$ on $L^2(\IR_n)$ defined by $D_Af({\t})=  d_A^{1/2}  f(A{\t})$.
   The translation of a
 function $f\in L^2(\IR^n)$ by ${\bf b}\in \IR^{n}$ will be denoted by
  $\tau_{{\bf
 b}}f(\t) = f({\t}-{\bf b}).$ For a subspace $S$ of $L^2(\IRd)$,
 $$
 {D}_{A}S= \{ {D}^j_{A}f : f \in S \}, \qquad \textrm{and} \qquad \tau_{{\bf
 b}}S= \{ \tau_{{\bf
 b}}f : f \in S \}.
 $$
 If we write ${D}^{-1}_{A} $ we mean the operator ${D}_{A^{-1}}$, ${D}^{0}_{A} $ is the identity map and  ${D}^{\ell}_{A} $, $\ell \in \mathbb{N} $, is the $\ell$-th composition of the operator ${D}_{A}$ with itself.

 If $A:\IR^n \to \IR^n$ is an expansive linear invertible map such that
$A(\IZ^n) \subset \IZ^n$, then the quotient groups  $\IZ^{n}/A(\IZ^{n})$ and $A^{-1}(\IZ^{n})/\IZ^{n}$ are well defined.
We will denote by  $\Omega_A \subset \IZ^n$ and $\Gamma_A \subset A^{-1}(\IZ^n)$ a full
collection of representatives of the cosets of  $\IZ^{n}/A(\IZ^{n})$ and $A^{-1}(\IZ^{n})/\IZ^{n}$ respectively.
 Recall that there are exactly $d_A$ cosets of $\IZ^{n}/A(\IZ^{n})$  (see \cite{GM:92} and \cite[p.109]{Wo:97}). Thus
 there are exactly $d_A$ cosets of $A^{-1}(\IZ^{n})/\IZ^{n}$.

For a  given $\phi \in L^2(\IR^n)$, set
\begin{equation*} \label{Phi}
[\phi,\phi](\t)= \sum_{{\bf k}\in \IZ^n}|\phi({\bf t}+
{\bf k})|^2
\end{equation*}
and denote
\begin{equation*}\label{d:2}
\mathcal{N}_{\phi} = \{ \t\in \IR^{n} ~:~  [\phi,\phi](\t) = 0 \}.
\end{equation*}

For a measurable function $f: \IR^n \to \IC$ the support
of $f$ is defined to be
$\textrm{supp}(f)= \{\mathbf{t} \in \IR^n ; f(\mathbf{t})\neq 0\} $.

The sets are defined modulo a null measurable set and we will
understand some equations as almost everywhere on
$\IR^n$ or $\IT^n$. Moreover, in order to shorten the notation, we will consider  $0/0
=0$ or $0(1/0) =0$ in some expressions where such an
indeterminacy appears.

The theory of frames was introduced by Duffin and
Schaeffer \cite{DS:52}.
A sequence $\{\phi_n \}_{n=1}^{\infty}$ of elements in a separable
Hilbert space ${\IH}$ is a {\it frame} for ${\IH}$ if there exist
constants  $C_1,C_2>0$ such that
\[
C_1\|h\|^2 \leq \sum_{n=1}^{\infty}|\langle h,\phi_n\rangle|^{2}\leq
C_2\|h\|^2, \qquad \forall h \in {\IH},
\]
where $\langle\cdot , \cdot \rangle$ denotes the inner product on $\IH$.
The constants $C_1$ and $C_2$ are called {\em frame bounds}.
The definition
implies that a frame is a complete sequence of elements of  ${\IH}$. A frame
$\{\phi_n \}_{n=1}^{\infty}$ is \emph{tight} if we may choose $C_1=C_2$,
and if in fact  $C_1=C_2=1,$ we will call the frame a
 {\em Parseval frame}.
A sequence $\{h_n \}_{n=1}^{\infty}$ of elements in a Hilbert space
${\IH}$ is a {\it  frame sequence}  if it is a frame for
${\overline{\mbox{span}}} \{h_n \}_{n=1}^{\infty}.$

Let $A : \IR^{n} \to \IR^n$ be an expansive linear map such that $A(\IZ^n) \subset \IZ^n$. Let $ \Psi=\{\psi_{1},\dots,\psi_{N} \} \subset L^2(\IR^n)$ be a set of functions, we call
\begin{eqnarray} \label{SWF}
X_{\Psi}:&= & \{D_A^j\tau_{{\k}}\psi_{\ell} ~ : ~  j \in \IZ, \ \k \in \IZd, \ 1 \le \ell \le N\}; \\
X^q_{\Psi}: & = & \{D_A^j\tau_{{\k}}\psi_{\ell} \, : \,  j \geq 0, \ \k \in \IZd, \ 1 \le \ell \le N \} \label{SWFq}\\ & & \bigcup \{ d_A^{j/2} \tau_{{\k}} D_A^j\psi_{\ell} \, : \,  j < 0, \ \k \in \IZd, \ 1 \le \ell \le N \} \nonumber
\end{eqnarray}
the affine system (resp. quasi-affine system) generated by $\Psi$. The set of functions
$ \Psi$
is called a {\em wavelet frame} associated to $A$, if the affine system \eqref{SWF}
is a frame for  $L^2(\IR^n)$. In this case, the affine system \eqref{SWF} is usually called {\em affine frame}.  When the  context is clear we do not write ``associated to $A$''. If this affine  system is a tight frame for
$L^2(\IR^n)$ then $\Psi$ is called a  {\em tight wavelet frame} or {\em tight framelet}. In particular, a tight wavelet
frame with frame bounds equal to $1$ is usually called a {\em Parseval wavelet frame} or a {\em Parseval framelet}.
The functions  $\psi_{\ell}$, $\ell =1, \dots ,N$ are called the {\em generators} of the wavelet frame. If the quasi-affine system \eqref{SWFq} is a frame for $L^2(\IR^n)$, then $X^q_{\Psi}$ is usually called {\em quasi-affine frame}.

Given a
 linear invertible map $A$ as above, by a frame multiresolution analysis associated to the dilation $A$ ($A$-FMRA) we mean  a
sequence of closed subspaces $V_j$, $j\in\IZ$, of the Hilbert space
$L^2(\IR^n)$ that satisfies the following conditions:
\begin{enumerate}
\item[(i)] $ V_j\subset V_{j+1}$ for every $j\in\IZ$;
\item[(ii)] $ f\in V_j$ if and only if  $D_Af \in V_{j+1}$ for every $j\in\IZ$;
\item[(iii)] $\overline{\bigcup_{j\in\IZ} V_j} = L^2(\IR^n)$;
\item[(iv)] There exists a function $\phi\in V_0$, that is called
\emph{scaling function}, such that  $V_0= \overline{\textrm{span} }\{ \tau_{\k}\phi  ~:~
{\bf k}\in\IZ^n \}$.
\end{enumerate}

According to Lemma E bellow, the  condition (iv) can be replaced by
\begin{enumerate}
\item[(iv')]  the system  $\{ \tau_{\k}\phi  ~:~
{\bf k}\in\IZ^n \}$ is a (Parseval) frame for $V_0$.
\end{enumerate}
When the system  $\{ \tau_{\k}\phi ~:~
{\bf k}\in\IZ^n \}$ is an orthonormal basis for $V_0$, then the $A$-FMRA is called an orthonormal multiresolution analysis or
simply a multiresolution analysis ($A$-MRA).

One  of the possible ways for constructing   an $A$-FMRA is to start with a scaling function $\phi \in L^2(\IR^n)$.
A  function $\phi \in  L^2(\IR^n)$ {\it generates an $A$-FMRA} if
$V_0= \overline{\textrm{span}}\{ \tau_{\k}\phi ~:~
{\bf k}\in\IZ^n \}$  and
the subspaces
\begin{equation}\label{c:1} V_j = {\overline{\mbox{span}}}
\{D_A^j\tau_{\k}\phi ~;~ {\k}\in \IZ^{n} \}, \quad j\in\IZ
\end{equation}
 of
the Hilbert space $L^2(\IR^n)$ satisfy the conditions (i) and (iii).

If  $\phi \in  L^2(\IR^n)$  generates an $A$-FMRA, then $\phi \in V_0 \subset V_1$. Thus,
\begin{equation} \label{refinabledefi}
 \phi=  \sum_{\mathbf{k}\in \mathbb{Z}^{n}}a_{\mathbf{k}} D_A\tau_{\k}\phi,
\qquad a_{\mathbf{k}} \in \mathbb{C},
\end{equation}
where the convergence is in $L^2(\mathbb{R}^n)$.
Taking the
Fourier transform, we can write
\begin{equation*} \label{refinable}
\widehat{\phi}(A^{\ast}\mathbf{t}) = H(\mathbf{t})
\widehat{\phi}(\mathbf{t}) \quad \textrm{   a.e. on $ \mathbb{R}^n$}
\end{equation*}
where $H$ is a $\IZ^n$-periodic measurable function.  A function that satisfies a equality as \eqref{refinabledefi}
is called refinable.

Given a multiresolution analysis, the following describes  a standard procedure for constructing wavelet frames.
If $\{V_j\}_{j \in \IZ} \subset L^2(\IR^n)$ is an $A$-FMRA,
we denote by $W_j$ the orthogonal complement of $V_j$ in $V_{j+1}$. Thus, by condition (i), we have  $V_{j+1} = W_j \oplus V_j$. Moreover,
condition (iii) implies that $V_{j+1} = \oplus_{k < j} W_k  $ and  $ L^2(\IR^d) = \oplus_{j \in \IZ} W_j$. Observe that if $\{ \tau_{\k}\psi_{\ell} \, : \,  \k \in \IZd, \ 1 \le \ell \le N\}$ is a Parseval frame  for $W_0$, then the system $\{D_A^j\tau_{{\k}}\psi_{\ell} \, : \,  j \in \IZ, \ \k \in \IZd, \ 1 \le \ell \le N\}$ is a Parseval  frame for $L^2(\IR^n)$.

\begin{defn}
Let $\{V_j\}_{j \in \IZ} \subset L^2(\IR^n)$ be an $A$-FMRA. A Parseval wavelet frame $\Psi= \{ \psi_1,\dots, \Psi_N\}$ for $L^2(\IR^n)$ is said to be associated with $\{V_j\}_{j \in \IZ}$ if $\{ \tau_{\k}\psi_{\ell} \, : \,  \k \in \IZd, \ 1 \le \ell \le N\}$ is a Parseval frame  for $W_0= V_1 \ominus V_0$. We say that a Parseval wavelet frame is an $A$-FMRA wavelet frame if  $\Psi$ is Parseval wavelet frame associated with some $A$-FMRA.
\end{defn}

A slight more flexible type of Parseval wavelet frames is the following.
\begin{defn} Let $\{V_j\}_{j \in \IZ} \subset L^2(\IR^n)$ be an $A$-FMRA. We say that $\Psi= \{ \psi_1,\dots, \Psi_N\} \subset L^2(\IR^n)$
is a Parseval wavelet frame arising from $\{V_j\}_{j \in \IZ}$ if $\Psi \subset V_1$ and the associated affine system \eqref{SWF}
is a Parseval wavelet frame for  $L^2(\IR^n)$. Sometimes it is said that  $\Psi= \{ \psi_1,\dots, \Psi_N\} \subset L^2(\IR^n)$
is an $A$-FMRA based wavelet frame if $\Psi$ is Parseval wavelet frame arising from some $A$-FMRA.
\end{defn}

A key tool in the study of wavelet frames is the Fourier transform.
Here, our convention is that  if $f\in L^1(\IR^n)$,
 \[
 \widehat{f}({\x}) := \int_{\mathbb{R}^n} f ({\t}) e^{-2\pi i{\t}\cdot {\x}} d{\t},
 \]
 where ${\t}\cdot {\x}$ denotes the usual inner product of vectors ${\t}$ and $ {\x}$ in $\IR^n$.
 The definition of Fourier transform is extended as usual form to  functions in $L^2(\IR^n)$.

The following definitions were introduced in \cite{CKS:05}.
\begin{defn} Let $A: \IR^n \to \IR^n$ be an expansive linear map. It is said  that ${\bf x}\in \IR^n$ is a point of
$A$-density for a set $E\subset\IR^n,$ $|E|_n>0$, if for any $r>0$
$$ \lim_{j\to \infty}\frac{|E\cap (A^{-j}B_r+{\bf
x})|_n}{|A^{-j}B_r|_n} =1.$$
\end{defn}
\begin{defn} Let $A: \IR^n \to \IR^n$ be an expansive linear map. Let $f: \IR^n \longrightarrow \IC$ be a measurable function.
It is said that ${\bf x}\in \IR^n$ is a point of $A$-approximate
continuity of the function $f$ if there exists $E\subset\IR^n,$
$|E|_n>0,$ such that ${\bf x}$ is a point of $A$-density for the set
$E$ and
\[
\lim_{{\scriptsize \begin{array}{ll}{\bf y} \to {\bf x} \\ {\bf
y}\in E
\end{array}} } f({\bf y}) = f({\bf x}).
\]
\end{defn}
\begin{defn} Let $A: \IR^n \to \IR^n$ be an expansive linear map.  A measurable function $f:\IR^n\rightarrow\IC$ is
said to be \emph{$A$-locally nonzero} at a point ${\bf x} \in \IR^n
$ if for any $\eps, r>0$  there exists $j\in \IN$ such that
\[ |\lb {\bf y} \in A^{-j}B_r+{\bf x} ~:~ f ({\bf y})=
0\rb|_n  < \eps | A^{-j}B_r|_n.
\]
\end{defn}

Observe that if $A=aI$, where $a >1 $ and $I$ is the identity map on $\IR^n$,
 the definition of a point of $A$-approximate continuity
coincides with the well-known definition  of \textit{approximate
continuity} (cf. \cite{N:60}, \cite{B:69}).

This paper is structured as follows. In Section \ref{Section2}, we write the main results of this paper, i.e., a
characterization of all Parseval wavelet frames can be constructed via  Oblique Extension Principle and a
characterization of all
Parseval wavelet frames  arising from a fix frame multiresolution analysis. In Section \ref{Section3},
we collect some well known results that we will use along this paper. The proofs of the main results of this paper will be postponed until Section \ref{Section4}.

\section{Main result} \label{Section2}

We write the main results of this paper.
In what follows, we fix  $A : \IR^n \to \IR^n $  an expansive linear map such that $A(\IZ^n) \subset \IZ^n$. Moreover
let us fix $\Gamma_{A^*}= \{\mathbf{p}_k \}_{k=0}^{d_A-1},$ where $\mathbf{p}_0=\mathbf{0}$,
 a full collection of representatives of the cosets of  $(A^*)^{-1}\IZ^{n}/\IZ^{n}$. In order to short the notation,
 if we write a wavelet frame we mean a wavelet frame associated to the dilation $A$.


The following result characterize all Parseval wavelet frames can be constructed via Oblique Extension Principle.
\begin{thm} \label{thm:main}
 Let $\phi \in L^2(\IR^n)$ such that
\[
\widehat{\phi}({A^*\t})=H_0({\t}) \widehat{\phi}({\t}), \quad \textrm{a.e.,}
\]
where $H_0 \in L^{\infty} (\IT^n)$. Let $H_{1}, \dots, H_N \in L^{\infty} (\IT^n)$ and define $\psi_{1},\dots, \psi_N \in L^2(\IR^n)$ by
\begin{equation} \label{psi}
\widehat{\psi_{\ell}}({A^*\t})=H_{\ell}({\t}) \widehat{\phi}({\t}) \quad \textrm{a.e.,} \qquad  \ell=1,\dots, N.
\end{equation}
Then the following are equivalent:
\begin{enumerate}
\item[i)]  The set of functions $\{ \psi_{\ell} \ : \ \ell=1,\dots, N \}$ is a Parseval wavelet frame  
for $L^2(\IR^n)$.
\item[ii)]  There exists $S$, a non-negative $\IZ^n$-periodic measurable function  such that $\sqrt{S}| \widehat{\phi} | \in L^2(\IR^n)$ and also\\
  $(a)$     the origin is a point of $A^*$-approximate continuity for $S|\widehat{\phi}|^2$, provided $S({\bf 0})|\widehat{\phi}({\bf 0})|^2=1$; \\
   $(b)$
    \begin{eqnarray} \label{SOEP0}
  S (A^*{\t}) | H_0({\t})|^2  +
	\sum_{\ell=1}^N |H_{\ell}({\t})|^2   = S({\t}) \quad a.e. \quad {\t} \in \IR^n \setminus \mathcal{N}_{\widehat{\phi}};
\end{eqnarray}\\
   $(c)$ the equality
   \begin{eqnarray} \label{SOEPk}
  S(A^*{\t}) H_0({\t}) \overline{H_0({\t} +  \mathbf{p}_k )} +
	\sum_{\ell=1}^N H_{\ell}({\t})\overline{H_{\ell}({\t}+ \mathbf{p}_k )}   = 0
\end{eqnarray}
holds for a.e. ${\t} \in \IR^n \setminus \mathcal{N}_{\widehat{\phi}}$ and for any $\mathbf{p}_k$, $k=1,\dots, d_A-1$, such that ${\t}+ \mathbf{p}_k  \in \IR^n \setminus \mathcal{N}_{\widehat{\phi}}$.
\end{enumerate}
\end{thm}

 The following result gives a characterization of all Parseval wavelet frames arising from a fix generalized
 multiresolution analysis.
 \begin{thm} \label{thm:main2}
 Let  $\{ V_j \}_{j\in \IZ} \subset L^2(\IR^n)$ be an $A$-FMRA with a scaling function $\phi $.  Let
  $\varphi \in L^2(\IR^n)$ defined by
  $\widehat{\varphi} ({\t})= \widehat{\phi} ({\t}) / [ \widehat{\phi},  \widehat{\phi} ]^{1/2} ({\t}) $.
  Let  $ \Psi = \{ \psi_{1},\dots, \psi_N \} \subset L^2(\IR^n)$. The following are equivalent.
\begin{enumerate}
\item[i)]  The set of functions $ \Psi$ is a Parseval wavelet frame arising from   $\{ V_j \}_{j\in \IZ}$.
\item[ii)]
 There exist $H_0, H_{1}, \dots, H_N \in L^{\infty} (\IT^n)$ such that \\
 $a)$
 \begin{equation*} \label{H0}
\widehat{\varphi}({A^*\t})=H_0({\t}) \widehat{\varphi}({\t}), \quad \textrm{a.e.,}
\end{equation*}
\begin{equation*} \label{Hl}
\widehat{\psi_{\ell}}({A^*\t})=H_{\ell}({\t}) \widehat{\varphi}({\t}) \quad \textrm{a.e.,} \qquad  \ell=1,\dots, N,
\end{equation*}
and \\
 $b)$ there exists  a non-negative $S\in L^{\infty}(\IT^n) $  such that  the origin is a point of $A^*$-approximate
 continuity for $S$ if we set  $S({\bf 0})=1$, and the equalities in \eqref{SOEP0} and \eqref{SOEPk} are satisfied.
\end{enumerate}
\end{thm}

A consequence of Thorem \ref{thm:main2} is the following characterization of all Parseval wavelet frames
associated with a fix frame multiresolution analysis.
 \begin{cor} \label{cor:main2}
 Let  $\{ V_j \}_{j\in \IZ} \subset L^2(\IR^n)$ be an $A$-FMRA with a scaling function $\phi $.  Let
  $\varphi \in L^2(\IR^n)$ defined by
  $\widehat{\varphi} ({\t})= \widehat{\phi} ({\t}) / [ \widehat{\phi},  \widehat{\phi} ]^{1/2} ({\t}) $.
  Let  $ \Psi = \{ \psi_{1},\dots, \psi_N \} \subset L^2(\IR^n)$. The following are equivalent.
\begin{enumerate}
\item[$\alpha$)]  The set of functions $ \Psi$ is a Parseval wavelet frame associated with   $\{ V_j \}_{j\in \IZ}$.
\item[$\beta$)] The condition ii) in Theorem  \ref{thm:main2} holds and also
  \[
  \sum_{k =0}^{d_A-1} H_0({\t}+  \mathbf{p}_k ) \overline{H_{\ell}({\t} +  \mathbf{p}_k )}  = 0, \quad \textrm{for a.e. } \,  {\t} \in \IR^n \setminus \mathcal{N}_{\widehat{\phi}}, \quad \ell= \{ 1, 2, \dots, N \}.
 \]
 \end{enumerate}
 \end{cor}

\section{Background} \label{Section3}

We collect  well known results on tight wavelet frames and scaling functions we will use to prove our main results in this paper.

We need the following characterization of Parseval wavelet frame for  $L^2(\IR^n)$ proved in \cite{B:00}.  Other versions appeared in
 \cite{FGWW:97}, \cite{RS1} and \cite{H:97}. \\
{\bf Theorem A.} {\em Suppose $\Psi= \{ \psi_{\ell} \ : \ \ell=1,\dots, N \} \subset L^2(\IR^n) $.
The following are equivalent:
\begin{enumerate}
\item[I)] The set of functions $\Psi$ is a Parseval wavelet frame for $L^2(\IR^n)$.
\item[II)]
$$
\sum_{\ell=1}^N \sum_{j \in \IZ} | \widehat{\psi_{\ell}} (A^{*j}{\t})|^2=1, \qquad a.e. \quad \textrm{ and}
$$
$$
\sum_{\ell=1}^N \sum_{j = 0}^{\infty} \widehat{\psi_{\ell}} (A^{*j}{\t})
 \overline{\widehat{\psi_{\ell}} (A^{*j}({\t} + \mathbf{q}))} =0,   \qquad a.e., \quad  \mathbf{q} \in \IZ^n \setminus A^*\IZ^n.
$$
\end{enumerate}}

The following is Theorem 2 (ii) in \cite{CSS:98}.\\
{\bf Theorem B.} {\em Let $\Psi$ be a finite set of functions in $L^2(\mathbb{R}^n)$. Then $X_{\Psi}$ is an affine frame if and only if $X^q_{\Psi}$ is a quasi-affine frame. Furthermore, their lower and upper exact frame bounds are equal.} \\

In  \cite{BW:94},  \cite{BL:98}, \cite{KL} and  \cite{CCK:01}
(see also \cite{Ch:03}), different versions of
a characterization of the functions in
$L^2(\mathbb{R} )$ whose integer translates generate a frame
sequence were given. Here we only write a version on $L^2(\mathbb{R}^n)$  because the proof
is completely similar to the case
$L^2(\mathbb{R})$.\\
{\bf Theorem C.}
{\em Let  $\phi\in L^2(\mathbb{R}^n)$. The system
$\{\tau_{\k} \phi  ~:~ \mathbf{k}\in \mathbb{Z}^{n}\} $ is a frame sequence with frame
bounds $C$ and $D$ if and only if
\begin{equation} \label{frameconstants}
C \leq [\widehat{\phi}, \widehat{\phi}] (\t) \leq D \quad \textrm{a.e. on $\IT^n
\setminus \mathcal{N}_{\widehat{\phi}}$.}
\end{equation}}

The following two results  can be found in  \cite{BVR:94} and \cite{BDR:93}.\\
{\bf Lemma D.} {\em Let $\phi \in L^2(\IR^n)$ and $V_0=\overline{\textrm{span}} \{ \tau_{\k}\phi ~:~ {\k}\in \IZ^n\}$ and
let $V_j:=D_A^jV_0$, $j\in \IZ$. A function $f\in L^2(\IR^n)$ is in
$V_j$ if and only if there exists $H$, an $\IZ^n$-periodic measurable function, such that
$\widehat{f}(A^{*j}{\t})= H(\t) \widehat{\phi}({\t})$ a.e. }\\
{\bf Lemma E.} {\em  Let  $\phi\in L^2(\IR^n)$ and  $V =
{\overline{\mbox{span}}} \{\tau_{\k}\phi ~:~ {\k}\in
\IZ^{n}\}. $ Then the system $\{\tau_{\k}\varphi  ~:~
{\k}\in \IZ^{n}\}$ is a Parseval frame for $V,$ where $\varphi \in
L^2(\IR^n)$ is the function defined by $\widehat{\varphi}=
\widehat{\phi} / [\widehat{\phi},\widehat{\phi}]^{\frac{1}{2}}$.}\\

Different versions of   the following lemma appeared in various publications
(cf. \cite{BDR:93}, \cite{BL:98}, \cite{Ch:03},  \cite{KS:08}). \\
{\bf Lemma F.}
{\em Let $\phi\in L^2(\mathbb{R}^n)$ and assume that
$\{\tau_{\k}\phi  ~:~ \mathbf{k}\in\mathbb{Z}^n \}$ is a frame
sequence in $L^2(\mathbb{R}^n).$
 If the subspaces $V_j,
j\in\mathbb{Z}$, are defined by ${(\ref{c:1})}$ then the
following conditions are equivalent:
\begin{enumerate}
\item[a)]  \qquad $\forall j\in\mathbb{Z},\qquad V_j\subset V_{j+1}$;
\item[b)] \qquad $V_0\subset V_{1}$;
\item[c)]  There exists a function $H \in L^\infty(\mathbb{T}^n)$ such that
  \begin{equation*}\label{cb:2}
\widehat{{\phi}}(A^*\mathbf{t})= H(
\mathbf{t})\widehat{\phi}(\mathbf{t}) \quad \mbox{a.e. on}\quad
\mathbb{R}^n.
\end{equation*}
\end{enumerate}}

The following proposition appeared in \cite{H:12} (cf. \cite{JS:94}, \cite{D:92}, \cite{BVR:94}).\\
{\bf Proposition G.}
{\em Let $\phi\in L^2(\mathbb{R}^d)$.
Then for each $f \in L^2(\IR^d)$ we have \\
$\lim_{j \to -\infty}  \sum_{{\k}\in \IZ^n } |\langle f, D_{A}^{j} \tau_{\k}\phi \rangle|^2 =0$. }\\

In a more general context, the following theorem was proved in \cite{KS:08}. That theorem is
formulated here in a modified form in order to indicate the essential result we need in this paper.  \\
{\bf Theorem H.}
{\em Let $V_j$, $j \in \IZ$, be a sequence of closed subspaces in $L^2(\IR^n)$ satisfying the conditions}
(i), (ii) {\em and} (iv) {\em with scaling function $\phi$}.  Then the following conditions are
equivalent:
\begin{enumerate}
\item[({\bf A})]  $\overline{\bigcup_{j\in \IZ}V_j}= L^2(\IR^n)$
\item[({\bf B})]
The function ${\widehat{\phi}}$  is $A^*$-locally nonzero at the origin;
\item[({\bf C})]   The origin is a
point of $A^*$-approximate continuity of the function
$|{\widehat{\phi}}|^2 / [\widehat{\phi},\widehat{\phi}], $
provided that $|{\widehat{\phi}({\bf 0})}|^2 / [\widehat{\phi},\widehat{\phi}]({\bf 0}) = 1$.
\end{enumerate}

Note that Lemma F and Theorem H together characterize all the functions $\phi \in L^2(\IR^n)$ that generate an $A$-FMRA.

The following proposition is a slight different version of Proposition 6 in \cite{SA:09}. \\
{\bf Proposition I.} {\em
Let $H_0 \in L^{\infty} (\IT^n)$ such that $|H_0 ({\bf t}) | \leq 1$ a.e.. Let $\phi \in L^2(\IR^n)$ such that
 the origin is a point of $A^*$-approximate continuity of  $| \widehat{\phi}|$, provided $| \widehat{\phi}({\bf 0})|=1$, and
\[
\widehat{\phi}({A^*\t})=H_0({\t}) \widehat{\phi}({\t}), \quad \textrm{a.e. ${\bf t} \in \IR^n$.}
\]
Let $\widetilde{H_0}({\t})= H_0({\t})$ if ${\t} \in \mathbb{R}^n \setminus \mathcal{N}_{\widehat{\phi}}$
and  $\widetilde{H_0}({\t})= 0$ if ${\t} \in  \mathcal{N}_{\widehat{\phi}}$.
Then
\[
| \widehat{\phi}(\mathbf{t}) | = \prod_{j=1}^{\infty} |
\widetilde{H_0}((A^{\ast})^{-j}\mathbf{t}) |, \quad \textrm{ a.e. ${\bf t} \in \IR^n$.}
\]}

We need the following auxiliary result on points of approximate continuity proved in \cite{SA:09}.\\
{\bf Lemma J.} {\em
Let  $A: \IR^n\rightarrow \IR^n$ be an expansive linear  map. Let $f:\IR^n \longrightarrow \IC$ be a measurable function
such that  for a point $\mathbf{y}\in \IR^n$ we have
\[
\lim_{j\longrightarrow \infty} f(A^{-j}\mathbf{x} +
\mathbf{y})=f(\mathbf{y}) \qquad \mbox{a.e. on $\IR^n$,}
\]
then the point $\mathbf{y}$ is a point of $A$-approximate continuity
of $f$.}\\

 The following technical
result is proved in  \cite{CKS:05}. Note that the equality (ii) in
the following lemma does not appear in the original result but it is
an immediate consequence of the proof of  (i).\\
{\bf Lemma K.} {\em Let $g\in L^2(\IT^n)$, let $A: \IR^n\rightarrow
\IR^n$ be a
 fixed linear invertible map  such that
 $A(\IZ^n) \subset \IZ^n$ and let $\hat{A}: \IT^n\rightarrow \IT^n$ be the induced
endomorphism. Let $\Gamma_A= \{\mathbf{q}_i \}_{i=0}^{d_A-1}$  be a full
collection of representatives of the cosets of  $A^{-1}(\IZ^{n})/\IZ^{n}$. Then
\begin{enumerate} \item[(i)]
$\int_{\IT^n} g(\hat{A}{\t}) d{\t} = \int_{\IT^n} g({\t}) d{\t},$
\item[(ii)]
$\int_{[0,1]^n} g({\t}) d{\t} = d_A^{-1}\int_{[0,1]^n}
\sum_{i=0}^{d_A-1}g(A^{-1}{\t} + \mathbf{q}_i) d{\t}.$
\end{enumerate}}

\section{Proofs of the main results} \label{Section4}

\subsection{Proof of Theorem 1.}

To prove  $ii)$ implies $i)$ in Theorem \ref{thm:main} we need the following results.

\begin{lem} \label{lem:02}  Let $\phi \in L^2(\IR^n)$ such that
$| \widehat{\phi}({\bf t})| \leq 1$ a.e. and the origin is a point of $A^*$-approximate continuity of
$| \widehat{\phi}|$, provided $| \widehat{\phi}({\bf 0})|=1$.  Let $f \in L^2(\IR^n)$ such that $\widehat{f}$ is continuous and  compactly supported.
Then, for any $ \varepsilon >0$ there exists $J \in \IN$ such that
\begin{equation} \label{eq:10}
(1- \varepsilon) \| f \|_2^2 \leq  \sum_{{\bf k} \in \IZ^n} | \langle  f, D_A^j \tau_{\bf k} \phi
\rangle   |^2  \leq  \| f \|_2^2, \qquad \forall \, j\geq J.
\end{equation}
\end{lem}
\begin{proof} Let  $f \in L^2(\IR^n)$ such that $\widehat{f}$ is continuous and
supp$(\widehat{f}) \subset B_R$ for a fix $R>0$. Let $K>0$ such that $|\widehat{f}({\bf t})| \leq K$
for every ${\bf t} \in \IR^n$. By Parseval's formula,
\begin{eqnarray} \label{sum}
\sum_{{\bf k} \in \IZ^n} | \langle  f, D_A^j \tau_{\bf k} \phi  \rangle   |^2 &= & \sum_{{\bf k} \in \IZ^n}
| \langle  \widehat{f}, \widehat{D_A^j \tau_{\bf k} \phi}  \rangle   |^2 = \sum_{{\bf k} \in \IZ^n}
| \langle D_{A^*}^j \widehat{f}, \widehat{ \tau_{\bf k} \phi}  \rangle   |^2 \nonumber \\
 &=&  \sum_{{\k} \in \IZd} |\int_{(A^*)^{-j}B_R} D_{A^*}^j\widehat{f}({\t}) \overline{\widehat{\phi}({\t})}
 e^{2 \pi {\k} \cdot {\t}} \,  d{\t} |^2.
\end{eqnarray}
Since $A^*$ is expansive, there exists $j_0 \in \mathbb{N}$ such that if $j \geq j_0$, then
$(A^*)^{-j}B_R \subset [-1/2,1/2]^d$. For those $j$, the
 sum over ${\k} \in \IZ^n$ in \eqref{sum} may be interpreted as the sum  of the square of the modulus of the $-{\k}$-th Fourier
coefficients of the function $D_{A^{*j}}\widehat{f}({\t}) \overline{\widehat{\phi}({\t})}$. Thus
\begin{eqnarray} \label{sum11}
\sum_{{\bf k} \in \IZ^n} | \langle  f, D_A^j \tau_{\bf k} \phi  \rangle   |^2 &= &  \int_{(A^*)^{-j}B_R}
|D_{A^{*j}} \widehat{f}({\t})|^2 | \widehat{\phi}({\t})|^2 \, d{\t}  \qquad \forall \ j  \geq j_0.
\end{eqnarray}
Since $| \widehat{\phi}({\t}) |\leq 1$ a.e., the right inequality of \eqref{eq:10} follows. Now, we prove the left
inequality of \eqref{eq:10}.

Let $0 < \varepsilon  < 1$ 
and take the set $\Lambda_{\varepsilon}= \{ \mathbf{t} \in \IR^n : |\widehat{\phi}({\bf t})| \leq 1- \frac{\varepsilon}{2} \}$.
Since $| \widehat{\phi}({\bf 0})|=1$, $| \widehat{\phi}({\bf t})| \leq 1$ a.e. and the origin is a point of
 $A^*$-approximate continuity of  $| \widehat{\phi}|$, then
$$
\lim_{j\to \infty} \frac{ |( (A^*)^{-j}B_R ) \cap \Lambda^c_{\varepsilon}|_n}{ |(A^*)^{-j}B_R |_n}=1.
$$
This implies that there exists $J \geq j_0 $ such that if $j \geq J$, we have
$$
| (A^*)^{-j}(B_R  \cap A^{*j}\Lambda_{\varepsilon})|_n  = |( (A^*)^{-j}B_R ) \cap \Lambda_{\varepsilon}|_n <
\dfrac{\varepsilon}{2 K^2 d_A^{j}} \| f \|_2^2.
$$
Thus, if $j \geq J$
\begin{equation} \label{approxiconti}
| B_R  \cap A^{*j}\Lambda_{\varepsilon}|_n    < \dfrac{\varepsilon}{2 K^2} \| f \|_2^2.
\end{equation}
According to \eqref{sum11}, if $j \geq J$ we obtain
\begin{eqnarray*}
\sum_{{\bf k} \in \IZ^n} | \langle  f, D_A^j \tau_{\bf k} \phi  \rangle   |^2 & \geq &
(1- \frac{\varepsilon}{2})  \int_{(A^*)^{-j}B_R \cap \Lambda^c_{\varepsilon}} |D_{A^*}\widehat{f}({\t})  |^2 \,
d{\t} \\ &=& (1- \frac{\varepsilon}{2}) \| f \|_2^2 - (1- \frac{\varepsilon}{2})  \int_{(A^*)^{-j}B_R \cap
\Lambda_{\varepsilon}} |D_{A^*} \widehat{f}({\t})  |^2 \, d{\t} \\ & \geq & (1- \frac{\varepsilon}{2}) \| f \|_2^2
-  \int_{B_R \cap A^{*j}\Lambda_{\varepsilon}} |\widehat{f}({\t})  |^2 \, d{\t}.
\end{eqnarray*}
Furthermore, by the inequality \eqref{approxiconti}, if $j \geq J$ we have
\begin{eqnarray*}
\sum_{{\bf k} \in \IZ^n} | \langle  f, D_A^j \tau_{\bf k} \phi  \rangle   |^2 & \geq & (1- \frac{\varepsilon}{2})
\| f \|_2^2 -   K^2 | B_R  \cap A^{*j}\Lambda_{\varepsilon}|_n \\ & \geq &  (1- \frac{\varepsilon}{2}) \| f \|_2^2
-   K^2 \dfrac{\varepsilon}{2 K^2} \| f \|_2^2 = (1- {\varepsilon}) \| f \|_2^2.
\end{eqnarray*}
This finishes the proof.
\end{proof}

We need the following
\begin{lem} \label{lem:03}
Let $\phi \in L^2(\IR^n)$ such that   $| \widehat{\phi}({\bf t})| \leq 1$ a.e. and
\[
\widehat{\phi}({A^*\t})=H_0({\t}) \widehat{\phi}({\t}), \quad \textrm{a.e.,}
\]
where $H_0 \in L^{\infty} (\IT^n)$. Let $H_{1}, \dots, H_N \in L^{\infty} (\IT^n)$ and define $\psi_{1},\dots, \psi_N \in L^2(\IR^n)$ by \eqref{psi}
Assume that
    \begin{eqnarray} \label{UEP0}
  | H_0({\t})|^2  +
	\sum_{\ell=1}^N |H_{\ell}({\t})|^2   = 1 \quad a.e. \, {\t} \in \IR^n \setminus \mathcal{N}_{\widehat{\phi}};
\end{eqnarray}
and
   the equality \begin{eqnarray} \label{UEPk}
 H_0({\t}) \overline{H_0({\t} +  \mathbf{p}_k )} +
	\sum_{\ell=1}^N H_{\ell}({\t})\overline{H_{\ell}({\t}+ \mathbf{p}_k )}   = 0
\end{eqnarray}
holds for a.e. ${\t} \in \IR^n \setminus \mathcal{N}_{\widehat{\phi}}$ and for any $\mathbf{p}_k$, $k=1,\dots, d_A-1$,
such that ${\t}+ \mathbf{p}_k  \in \IR^n \setminus \mathcal{N}_{\widehat{\phi}}$.
Then, for all $j \in \IZ$ and for all $f \in L^2(\IR^n)$ such that $\widehat{f}$ is continuous and  compactly supported, we have
\begin{equation} \label{sumDt}
\sum_{\k \in \IZ^n} | \langle f, D_A^j \tau_{\k}\phi \rangle|^2 =
\sum_{\k \in \IZ^n} | \langle f, D_A^{j-1} \tau_{\k}\phi \rangle|^2 + \sum_{\ell=1}^N \sum_{\k \in \IZ^n} | \langle f, D_A^{j-1}
\tau_{\k}\psi_{\ell} \rangle|^2.
\end{equation}
\end{lem}
\begin{proof}
Let $f \in L^2(\IR^n) $ such that $\widehat{f}$ is continuous and  compactly supported. Fix a $j \in \IZ$, ${\k} \in \IZ^n$  and
 $\ell \in \{ 1, \cdots, N \}$. Using the Parseval's equality, the change of variable $A^{*-j}{\t}= {\bf s}$ and \eqref{psi}, we can write
\begin{eqnarray*}
\langle f, D_{A}^{j-1} \tau_{\k}\psi_{\ell} \rangle & = & \langle \widehat{f}, \widehat{D_{A}^{j-1}
\tau_{\k}\psi_{\ell}} \rangle \\
&=& \int_{\IR^n} d_A^{j/2} \widehat{f}(A^{*j}{\bf s}) d_A^{1/2} \overline{\widehat{\psi_l}(A^*{\bf s})}
e^{2 \pi i {\k} \cdot A^*{\bf s}} \, d{\bf s} \\
&=& \sum_{{\bf m}\in \IZ^n}\int_{[0,1)^n-{\bf m}} d_A^{j/2} \widehat{f}(A^{*j}{\bf s}) d_A^{1/2}
\overline{H_{\ell}({\bf s})\widehat{\phi}({\bf s})} e^{2 \pi i {\k}\cdot A^*{\bf s}} \, d{\bf s}.
\end{eqnarray*}
Since $\widehat{f}$ has  compact support, the above sum only involves a finite number of  ${\bf m}$'s. Moreover,
bearing in mind that $H_{\ell}$ is a $\IZ^n$-periodic function and $A^*(\IZ^n) \subset \IZ^n$, we have
\begin{equation*}
\langle f, D_{A}^{j-1} \tau_{\k}\psi_{\ell} \rangle
= \int_{[0,1)^n} d_A^{1/2} \overline{H_{\ell}({\bf s})} \left(  \sum_{{\bf m}\in \IZ^n} \tau_{{\bf m}} \left(
 (D_{A^*}^j \widehat{f} )  \overline{\widehat{\phi}} \right)  ({\bf s}) \right)  e^{2 \pi i {\k}\cdot A^*{\bf s}} \, d{\bf s}.
\end{equation*}
According to our hypotheses, $\widehat{f}$, $H_{\ell}$ and $\widehat{\phi}$ are bounded, then the function into the above integral
is in $L^2(\IT^n)$. By (ii) in Lemma K, we have
\begin{eqnarray*}
&  & \langle f, D_{A}^{j-1} \tau_{\k}\psi_{\ell} \rangle \\
&=& \int_{[0,1)^n} d_A^{-1/2}\sum_{k=0}^{d_A-1}  \overline{H_{\ell}((A^*)^{-1}{\bf s} + \mathbf{p}_k )} \\
& & \hspace{1cm} \times \left(  \sum_{{\bf m}\in \IZ^n} \tau_{{\bf m}} \left( (D_{A^*}^j \widehat{f} )  \overline{\widehat{\phi}} \right)
 ((A^*)^{-1}{\bf s} + \mathbf{p}_k) \right)  e^{2 \pi i {\k}\cdot {\bf s}} \, d{\bf s}
\end{eqnarray*}
which is the $-{\k}$-th Fourier coefficient of the function
$$
d_A^{-1/2}\sum_{k=0}^{d_A-1}  \overline{H_{\ell}((A^*)^{-1}{\bf s} + \mathbf{p}_k )}
 \left(  \sum_{{\bf m}\in \IZ^n} \tau_{{\bf m}} \left( (D_{A^*}^j \widehat{f} )  \overline{\widehat{\phi}} \right)  ((A^*)^{-1}{\bf s}
 + \mathbf{p}_k) \right).
$$
With analogous computation for $\langle f, D_{A}^{j-1} \tau_{\k}\phi \rangle$, we get
\begin{eqnarray*}
&  & \sum_{{\k}\in \IZ^n } |\langle f, D_{A}^{j-1} \tau_{\k}\phi \rangle|^2 +
\sum_{\ell=1}^N\sum_{{\k}\in \IZ^n } |\langle f, D_{A}^{j-1} \tau_{\k}\psi_{\ell} \rangle|^2 \\
&=& d_A^{-1} \int_{[0,1)^n} \sum_{\ell=0}^N \Big| \sum_{k=0}^{d_A-1}  \overline{H_{\ell}(A^*)^{-1}({\bf s} + \mathbf{p}_k )}  \\
  & &  \hspace{1cm } \times \big(  \sum_{{\bf m}\in \IZ^n} \tau_{{\bf m}} \left( (D_{A^*}^j \widehat{f} )
  \overline{\widehat{\phi}} \right)  ((A^*)^{-1}{\bf s} + \mathbf{p}_k) \big) \Big|^2 \, d{\bf s}.
\end{eqnarray*}
Bearing in mind  the functions inside the integral are $\IZ^n$-periodic, we obtain
\begin{eqnarray*}
&  & \sum_{{\k}\in \IZ^n } |\langle f, D_{A}^{j-1} \tau_{\k}\phi \rangle|^2 +
\sum_{\ell=1}^N\sum_{{\k}\in \IZ^n } |\langle f, D_{A}^{j-1} \tau_{\k}\psi_{\ell} \rangle|^2 \\
&=& d_A^{-1} \int_{[0,1)^n} \sum_{\ell=0}^N  \sum_{k=0}^{d_A-1} | {H_{\ell}((A^*)^{-1}{\bf s} + \mathbf{p}_k )} |^2 \\
 & &  \hspace{1cm } \times \left|  \sum_{{\bf m}\in \IZ^n} \tau_{{\bf m}} \left( (D_{A^*}^j \widehat{f} )
  \overline{\widehat{\phi}} \right)  ((A^*)^{-1}{\bf s} + \mathbf{p}_k) \right|^2 \, d{\bf s}\\
 & & +   d_A^{-1} \int_{[0,1)^n}   \sum_{k =0}^{d_A-1}  \sum_{a=1}^{d_A-1} \left( \sum_{\ell=0}^N
 H_{\ell}((A^*)^{-1}{\bf s} + \mathbf{p}_k ) \overline{H_{\ell}((A^*)^{-1}{\bf s} + \mathbf{p}_k + \mathbf{p}_a )}
 \right) \\  &  &  \hspace{1cm}  \times     \left( \sum_{{\bf m}\in \IZ^n} \tau_{{\bf m}} \left(
  (D_{A^*}^j \widehat{f} )  \overline{\widehat{\phi}} \right)  ((A^*)^{-1}{\bf s} + \mathbf{p}_k) \right)\\
  &  &  \hspace{2cm}  \times  \left( \sum_{{\bf b}\in \IZ^n} \tau_{{\bf b}} \left( \overline{(D_{A^*}^j \widehat{f} )}
   {\widehat{\phi}} \right)  ((A^*)^{-1}{\bf s} + \mathbf{p}_k + \mathbf{p}_a) \right) \, d{\bf s}.
\end{eqnarray*}

By \eqref{UEP0} and \eqref{UEPk}
we obtain
\begin{eqnarray*}
&  & \sum_{{\k}\in \IZ^n } |\langle f, D_{A^*}^{j-1} \tau_{\k}\phi \rangle|^2 + \sum_{\ell=1}^N\sum_{{\k}\in \IZ^n }
|\langle f, D_{A^*}^{j-1} \tau_{\k}\psi_{\ell} \rangle|^2 \\
&=& d_A^{-1} \int_{[0,1)^n}  \sum_{k =0}^{d_A-1} \left|  \sum_{{\bf m}\in \IZ^n} \tau_{{\bf m}} \left( (D_{A^*}^j \widehat{f} )
\overline{\widehat{\phi}} \right)  ((A^*)^{-1}{\bf s} + \mathbf{p}_k) \right|^2 \, d{\bf s}.
\end{eqnarray*}
Since the sum over ${\bf m}$ is  finite and the funtions  $\widehat{f}$ and $\widehat{\phi}$ are bounded,\\
 $\left|  \sum_{{\bf m}\in \IZ^n} \tau_{{\bf m}} \left( (D_{A^*}^j \widehat{f} )  \overline{\widehat{\phi}} \right)  \right|^2$
 is in $L^2(\IT^n)$. Thus, by (ii) in Lemma K we obtain
\begin{eqnarray*}
&  & \sum_{{\k}\in \IZ^n } |\langle f, D_{A}^{j-1} \tau_{\k}\phi \rangle|^2 +
\sum_{\ell=1}^N\sum_{{\k}\in \IZ^n } |\langle f, D_{A}^{j-1} \tau_{\k}\psi_{\ell} \rangle|^2 \\
&=&  \int_{[0,1)^n}   \left|  \sum_{{\bf m}\in \IZ^n} \tau_{{\bf m}} \left( (D_{A^*}^j \widehat{f} )
 \overline{\widehat{\phi}} \right)  ({\bf s}) \right|^2 \, d{\bf s} =   \sum_{{\k}\in \IZ^n }
  |\langle f, D_{A}^{j} \tau_{\k}\phi \rangle|^2,
\end{eqnarray*}
as we wanted to prove.
\end{proof}

We are ready to prove the following  version of Unitary Extension Principle.

\begin{thm}[Unitary Extension Principle] \label{thm:UEP}  Let $\phi \in L^2(\IR^n)$ such that
 the origin is a point of $A^*$-approximate continuity of  $| \widehat{\phi}|$, provided $| \widehat{\phi}({\bf 0})|=1$, and
\[
\widehat{\phi}({A^*\t})=H_0({\t}) \widehat{\phi}({\t}), \quad \textrm{a.e.,}
\]
where $H_0 \in L^{\infty} (\IT^n)$. Let $H_{1}, \dots, H_N \in L^{\infty} (\IT^n)$ such that \eqref{UEP0} and \eqref{UEPk} hold.
If  $\psi_{1},\dots, \psi_N \in L^2(\IR^n)$ are defined by \eqref{psi}, then
$\{ \psi_{\ell} \ : \ \ell=1,\dots, N \}$  is a Parseval wavelet frame for $L^2(\IR^n)$.
\end{thm}
\begin{proof}
By \eqref{UEP0}, it is obvious that $|H_0 ({\bf t}) | \leq 1$ a.e. Thus, according to Proposition~I, we have
 that $| \widehat{\phi}({\bf t})|\leq 1$ a.e.

 Let $\varepsilon >0$ be given and let $f \in L^2(\IR^n) $ such that $\widehat{f}$ is continuous and  compactly supported.
 For $j_0 \in \IZ$, Lemma \ref{lem:03}  shows that
$$ \sum_{\k \in \IZ^n} | \langle f, D_{A}^{j_0} \tau_{\k}\phi \rangle|^2 =
\sum_{\k \in \IZ^n} | \langle f, D_{A}^{j_0-1} \tau_{\k}\phi \rangle|^2 + \sum_{\ell=1}^N \sum_{\k \in \IZ^n}
| \langle f, D_{A}^{j_0-1} \tau_{\k}\psi_{\ell} \rangle|^2.
$$
Repeating the argument for $j_0-1, j_0-2,\cdots$, it follows that if $j < j_0$ we obtain
\begin{equation*} \sum_{\k \in \IZ^n} | \langle f, D_{A}^{j_0} \tau_{\k}\phi \rangle|^2 =
\sum_{\k \in \IZ^n} | \langle f, D_{A}^{j} \tau_{\k}\phi \rangle|^2 + \sum_{\ell=1}^N \sum_{m=j}^{j_0-1 }\sum_{\k \in \IZ^n}
| \langle f, D_{A}^{m} \tau_{\k}\psi_{\ell} \rangle|^2.
\end{equation*}
Then, by Lemma \ref{lem:02},
 there exists $J \in \IN$ such that if $j_0\geq J$ and $j < j_0$ we have
\begin{equation} \label{acota}
(1- \varepsilon) \| f \|_2^2 \leq  \sum_{\k \in \IZ^n} | \langle f, D_{A}^{j} \tau_{\k}\phi \rangle|^2
+ \sum_{\ell=1}^N \sum_{m=j}^{j_0-1 }\sum_{\k \in \IZ^n} | \langle f, D_{A}^{m} \tau_{\k}\psi_{\ell} \rangle|^2 \leq
 \| f \|_2^2.
\end{equation}
By Proposition G we know that
$$
\lim_{j \to - \infty} \sum_{\k \in \IZ^n} | \langle f, D_{A}^{j} \tau_{\k}\phi \rangle|^2=0.
$$
Therefore, letting $j \to - \infty$ in \eqref{acota}, for $j_0\geq J$ we have
\begin{equation*}
(1- \varepsilon) \| f \|_2^2 \leq  \sum_{\ell=1}^N \sum_{m=- \infty}^{j_0-1 }\sum_{\k \in \IZ^n}
 | \langle f, D_{A}^{m} \tau_{\k}\psi_{\ell} \rangle|^2 \leq  \| f \|_2^2.
\end{equation*}
In addition,  letting $j_0 \to  \infty$,
\begin{equation*}
(1- \varepsilon) \| f \|_2^2 \leq  \sum_{\ell=1}^N \sum_{m=- \infty}^{\infty }\sum_{\k \in \IZ^n}
| \langle f, D_{A}^{m} \tau_{\k}\psi_{\ell} \rangle|^2 \leq  \| f \|_2^2.
\end{equation*}
Since $\varepsilon >0$ is arbitrary, we have
\begin{equation*}
  \sum_{\ell=1}^N \sum_{m=- \infty}^{\infty }\sum_{\k \in \IZ^n} | \langle f, D_{A}^{m} \tau_{\k}\psi_{\ell}
  \rangle|^2 = \| f \|_2^2.
  \end{equation*}
  The proof is finished by a density argument.
\end{proof}

A more flexible way for constructing wavelet frames is the following result.
\begin{thm}[Oblique Extension Principle] \label{thm:OEP}
  Let $\phi \in L^2(\IR^n)$ such that
\[
\widehat{\phi}({A^*\t})=H_0({\t}) \widehat{\phi}({\t}), \quad \textrm{a.e.,}
\]
where $H_0 \in L^{\infty} (\IT^n)$. Let $H_{1}, \dots, H_N \in L^{\infty} (\IT^n)$ and define $\psi_{1},\dots, \psi_N \in L^2(\IR^n)$
by \eqref{psi}.
Assume that there exists $S$ a non-negative $\IZ^n$-periodic measurable function  such that $\sqrt{S}| \widehat{\phi} | \in L^2(\IR^n)$,
 the origin is a point of $A^*$-approximate continuity of $S| \widehat{\phi} |^2 $, provided $S({\bf 0})| \widehat{\phi}({\bf 0}) |^2=1$.
 Moreover \eqref{SOEP0} and \eqref{SOEPk} hold. Then  the set of functions $\{ \psi_{\ell} \ : \ \ell=1,\dots, N \}$ is a
 Parseval wavelet frame for $L^2(\IR^n)$.
\end{thm}
\begin{proof}
Assume that the conditions of Theorem \ref{thm:OEP} are satisfied. Define $\varphi \in L^2(\IR^n)$ by
\begin{equation} \label{phi111111}
\widehat{\varphi} ({\t})= \sqrt{S({\t})}\widehat{\phi} ({\t})
\end{equation}
and define $Q_0, Q_1 \dots, Q_N$, $\IZ^n$-periodic measurable functions, by
$$
Q_0({\t})= \sqrt{\frac{S(A^*{\t})}{S({\t})}} H_0({\t}), \qquad Q_{\ell}({\t})= \sqrt{\frac{1}{S({\t})}} H_{\ell}({\t}),
\quad  \ell= 1,\dots, N.
$$
Observe that according to $\eqref{SOEP0}$, we have $Q_0, Q_1 \dots, Q_N \in L^{\infty}(\IT^n)$. In addition, if $S({\t})=0$ then $ H_{\ell}({\t})=0$, $\ell=1,\dots, N$ and $S(A^*{\t}) H_0({\t})=0$.
In this case we use the convention $0/0=0$ in the definitions of $Q_0, Q_1 \dots, Q_N$.

Now, we show that $\varphi$ and $Q_0, Q_1 \dots, Q_N $ satisfy the hypotheses of Theorem \ref{thm:UEP}.
First,  the origin is a point of $A^*$-approximate continuity for $ |\widehat{\varphi}|^2$ if we set
 $ |\widehat{\varphi}({\bf 0})|^2=1$ by hypothesis.

Second,
\begin{eqnarray*}
\widehat{\varphi} (A^*{\t}) &=& \sqrt{S(A^*{\t})}\widehat{\phi} (A^*{\t}) = \sqrt{S(A^*{\t})} H_0({\t}) \widehat{\phi} ({\t})= Q_0({\t}) \widehat{{\varphi}} ({\t}) \qquad a.e.
\end{eqnarray*}

Third, by the definition of $Q_{\ell}$ and \eqref{SOEP0}, we obtain
\begin{equation*} \label{eq:111}
\sum_{\ell=0}^N |Q_{\ell}({\t})|^2 = \frac{S(A^*{\t})}{S({\t})} |H_0({\t})|^2 +
\sum_{\ell=1}^N \frac{1}{S({\t})} |H_{\ell}({\t})|^2 =1 \quad \textrm{a.e. ${\t} \in \IR^n \setminus \mathcal{N}_{\widehat{\varphi}}$.}
\end{equation*}
Moreover, by \eqref{SOEPk}, bearing in mind that  $S$ is a $\IZ^n$-periodic functions and $A^*\mathbf{p}_k \in \IZ^n$, we also have
\begin{eqnarray*} \label{eq:1112}
\sum_{\ell=0}^N Q_{\ell}({\t})  \overline{Q_{\ell}({\t} + \mathbf{p}_k)} &= &
 \frac{S(A^*{\t})}{\sqrt{S({\t}) S ({\t} + \mathbf{p}_k)}} H_{0}({\t})  \overline{H_{0}({\t}
 +\mathbf{p}_k)}\\ & + &   \sum_{\ell=1}^N \frac{1}{\sqrt{S({\t}) S ({\t} +
 \mathbf{p}_k)}} H_{\ell}({\t})  \overline{H_{\ell}({\t} + \mathbf{p}_k)} =0,
\end{eqnarray*}
for a.e. ${\t} \in \IR^n \setminus \mathcal{N}_{\widehat{\varphi}}$ and for any $\mathbf{p}_k$,
 $k=1,\dots, d_A-1$, such that ${\t}+ \mathbf{p}_k  \in \IR^n \setminus \mathcal{N}_{\widehat{\varphi}}$.

 Since we have seen the hypotheses of Theorem \ref{thm:UEP}, we conclude that the set functions
  $\{ \widetilde{\psi_{\ell}} \ : \ \ell=1, \dots, N \}$, where
$$
\widehat{ \widetilde{\psi_{\ell}}} (A^*{\t}) = Q_{\ell}({\t}) \widehat{\varphi} ({\t}), \qquad a.e.
$$
is a Parseval wavelet frame for $L^2(\IR^n)$. Indeed, by the observation that
$$
\widehat{ \widetilde{\psi_{\ell}}} (A^*{\t}) = Q_{\ell} ({\t}) \widehat{\varphi} ({\t}) =
\sqrt{\frac{1}{S({\t})}} H_{\ell}({\t}) \sqrt{S({\t})}\widehat{\phi} ({\t})= \widehat{ \psi_{\ell}} (A^*{\t}),
 \qquad a.e.
$$
the proof is completed.
\end{proof}

A key tool to prove $i)$ implies $ii)$ in Theorem \ref{thm:main} is the fundamental function.
 Let $\phi \in L^2(\IR^n)$ such that
\[
\widehat{\phi}({A^*\t})=H_0({\t}) \widehat{\phi}({\t}), \quad \textrm{a.e.,}
\]
where $H_0$ is a finite a.e.,  $\IZ^n$-periodic measurable function such that $H_0({\t})=0$ a.e.
${\t} \in  \mathcal{N}_{\widehat{\phi}} $. Moreover, let $\psi_{1},\dots, \psi_N \in L^2(\IR^n)$ defined by
\begin{equation} \label{psi11}
\widehat{\psi_{\ell}}({A^*\t})=H_{\ell}({\t}) \widehat{\phi}({\t}) \quad \textrm{a.e.,} \qquad  \ell=1,\dots, N,
\end{equation}
where $H_{\ell}$ is a finite a.e., $\IZ^n$-periodic measurable function such that $H_{\ell}({\t})=0$ a.e.
${\t} \in  \mathcal{N}_{\widehat{\phi}} $.
The following
 \begin{equation} \label{fundafun}
 \Theta({\t})= \sum_{m=0}^{\infty} \sum_{\ell=1}^N |H_{\ell} (A^{*m}{\t})|^2 \prod_{k=0}^{m-1} |H_{0} (A^{*k}{\t})|^2,
 \end{equation}
 with  the convention $\prod_{k=0}^{-1} |Q_{0} (A^{*k}{\t})|^2=1$, is usually called {\em fundamental function}
 associated to $H_0, \dots, H_N$. Note that fundamental functions were introduced in \cite{RS1}. Assuming that
  $\{ \psi_1,\dots, \psi_N \}$ is a Parseval wavelet frame, we will focus on properties $\Theta$, for instance,
   we will see that $\Theta$ is a finite a.e. measurable function.

\begin{proof}[Proof of Theorem \ref{thm:main}]

That  $ii)$ implies $i)$ follows by  Oblique Extension Principle.

We prove i) implies $ii)$. Consider
${Q_{\ell}}$, $\ell=0,\dots, N$, defined as $Q_{\ell}(\t)=H_{\ell}(\t)$,   a.e. on
${\t} \in\IR^n \setminus  \mathcal{N}_{\widehat{\phi}}$ and ${Q_{\ell}}(\t)=0$,  a.e.
on ${\t} \in  \mathcal{N}_{\widehat{\phi}}$.

Observe that
\begin{equation} \label{phiQ}
\widehat{\phi}({A^*\t})=Q_0({\t}) \widehat{\phi}({\t}), \quad \textrm{a.e.,}
\end{equation}
and
\begin{equation} \label{psiQ}
\widehat{\psi_{\ell}}({A^*\t})=Q_{\ell}({\t}) \widehat{\phi}({\t}) \quad \textrm{a.e.,} \qquad  \ell=1,\dots, N.
\end{equation}
Let $\Theta$ be the fundamental function associated to $Q_0,\dots, Q_N$.
We will see that  $\Theta$ is a  non-negative $\IZ^n$-periodic measurable function satisfying
$\sqrt{\Theta}| \widehat{\phi} | \in L^2(\IR^n)$, (a), (b)  and   (c) of ii) if we consider $\Theta$ instead of $S$ in those conditions.

  First, we show that $\Theta$ is a non-negative, $\IZ^n$-periodic measurable function such that
  $\sqrt{\Theta}| \widehat{\phi} | \in L^2(\IR^n)$. Let $f \in L^2(\IR^n)$.
  Since $\{ \psi_{\ell} \ : \ \ell=1,\dots, N \}$ is a Parseval wavelet frame for $L^2(\IR^n)$ and by Theorem B,  we have
$$
\sum_{j=1}^{\infty} \sum_{{\k} \in \IZ^n} \sum_{\ell=1}^N |\langle f, d_A^{-j/2} \tau_{{\k}} D_A^{-j}\psi_{\ell}\rangle|^2
\leq  \|f \|_2^2,
$$
or equivalently
$$
\sum_{j=1}^{\infty}  \sum_{\ell=1}^N  \int_{[0,1]^n} | \sum_{{\k} \in \IZ^n}
 \widehat{f}({\bf s}- {\k})  \overline{\widehat{\psi_{\ell}}(A^{*j}({\bf s}-{\k}))}|^2 \, d{\bf s} \leq  \|f \|_2^2.
$$
 Bearing in mind that $Q_{\ell}$ is $\IZ^n$-periodic, $A^*(\IZ^n) \subset \IZ^n$,
\eqref{psiQ} and \eqref{phiQ}, we obtain
$$
  \int_{[0,1]^n} \Theta({\t}) | \sum_{{\k} \in \IZ^n}  \widehat{f}({\bf t}- {\k})
   \overline{\widehat{\phi}({\bf t}-{\k})}|^2 \, d{\bf t} \leq  \|f \|_2^2.
$$
Therefore, $\Theta  \sum_{{\k} \in \IZ^n}  | \widehat{\phi}(\cdot-{\k})|^2$ is in $L^{\infty}(\IT^n)$.
This implies that
$\sqrt{\Theta} \widehat{\phi} \in L^{2}(\IR^n)$ and  $\Theta$ is finite a.e..  The measurability of  $\Theta$
holds because it is
defined as the pointwise limit of measurable functions a.e..
The  $\IZ^n$-periodicity of $\Theta$ follows by its definition.

We check (a) in $ii)$. 
According to Theorem A, the definition of $\widehat{\psi_{\ell}}$ and the refinement equation associated to $\phi$,
\begin{eqnarray*}
1 &= & \sum_{j \in \IZ}  \sum_{\ell=1}^N  | \widehat{\psi_{\ell}} (A^{*j}{\t})|^2 =\lim_{J\to - \infty}
\sum_{j= J+1}^{\infty}  \sum_{\ell=1}^N  | \widehat{\psi_{\ell}} (A^{*j}{\t})|^2  \\
 & & \lim_{J\to - \infty} \sum_{j= J}^{\infty}  ( \sum_{\ell=1}^N  |Q_{\ell} (A^{*j}{\t})|^2
 \prod_{k=J}^{j-1} |Q_{0} (A^{*k}{\t})|^2 ) |\widehat{\phi}(A^{*J}{\t}) |^2 \\ &=&
 \lim_{J\to - \infty} \Theta (A^{*J}{\t}) |\widehat{\phi}(A^{*J}{\t}) |^2 \qquad a.e..
\end{eqnarray*}
Thus, the origin is a point of $A^*$-approximate continuity of  $\Theta |\widehat{\phi} |^2$, if we
set $\Theta ({\bf 0}) |\widehat{\phi}({\bf 0}) |^2 =1$, follows  by Lemma J.

To  prove  the condition (b) in $ii)$, observe that from the definition of $\Theta$, we have
\begin{eqnarray*}
\Theta(\t) &= & \sum_{m=0}^{\infty} \sum_{\ell=1}^N |Q_{\ell} (A^{*m}{\t})|^2 \prod_{k=0}^{m-1} |Q_{0} (A^{*k}{\t})|^2 \\
&= &  \Theta(A^*\t) |Q_{0} ({\t})|^2+  \sum_{\ell=1}^N |Q_{\ell} ({\t})|^2 .
\end{eqnarray*}
Since $Q_{\ell} ({\t})=H_{\ell} ({\t})$ a.e. on ${\t} \in\IR^n \setminus  \mathcal{N}_{\widehat{\phi}}$, we conclude that
 $\Theta$ satisfies the condition (b).

We see (c) in $ii)$. Let $k \in \{ 1,\dots, d_A-1 \}$.  For almost  $(A^*)^{-1}{\t}$ and $ (A^*)^{-1}{\t} + \mathbf{p}_k$
points that are in $ \IR^n \setminus \mathcal{N}_{\widehat{\phi}}$, then there exist ${\k}_1, {\k}_2 \in \IZ^n$
such that $ \widehat{\phi}((A^*)^{-1}{\t} + {\bf k}_1)  \widehat{\phi}((A^*)^{-1}{\t} + {\bf k}_2 + \mathbf{p}_k) \neq 0$.
Let  ${\bf q} =  A^*({\bf k}_2- {\bf k}_1)+ A^*\mathbf{p}_k$ and observe that
$\mathbf{q} \in \IZ^n \setminus A^*\IZ^n $.   We call $(A^*)^{-1}{\bf s}=(A^*)^{-1}{\t} + {\bf k}_1$.
 By Theorem A, the definition of $\widehat{\psi_{\ell}}$, the refinement equation
associated to $\phi$ and the definition of the fundamental function $\Theta$, we have
\begin{eqnarray*}
0 &= & \sum_{\ell=1}^N \sum_{j = 0}^{\infty} \widehat{\psi_{\ell}} (A^{*j}{\bf s})
 \overline{\widehat{\psi_{\ell}} (A^{*j}({\bf s} + \mathbf{q}))}  \\
 &=&  \widehat{\phi }((A^*)^{-1}{\bf s}) \overline{\widehat{\phi }((A^*)^{-1}{\bf s} + (A^*)^{-1}\mathbf{q})} \\
  & & \times
 \Big( \Theta({\bf s})  Q_{0} ((A^*)^{-1}{\bf s})\overline{ Q_{0} ((A^*)^{-1}{\bf s} + (A^*)^{-1}\mathbf{q}) }  \\
 & &  \hspace{0.5cm}  + \sum_{\ell=1}^N
 Q_{\ell} ((A^*)^{-1}{\bf s }) \overline{Q_{\ell} ((A^*)^{-1}{\bf s} + (A^*)^{-1}\mathbf{q})} \Big),
\end{eqnarray*}
Bearing in mind that  $Q_{\ell} ({\t})=H_{\ell} ({\t})$ a.e. on ${\t} \in\IR^n \setminus  \mathcal{N}_{\widehat{\phi}}$
and they are $\IZ^n$-periodic functions, we have
\begin{equation*}
0  =
  \Theta({\bf t})  H_{0} ((A^*)^{-1}{\bf t}) \overline{H_{0} ((A^*)^{-1}{\bf t} + \mathbf{p}_k) }    + \sum_{\ell=1}^N
 H_{\ell} ((A^*)^{-1}{\bf t}) \overline{H_{\ell} ((A^*)^{-1}{\bf t} + \mathbf{p}_k)}. 
\end{equation*}
Hence, the condition (c) in $ii)$ of Theorem~\ref{thm:main} follows.
Thus the proof is finished.
\end{proof}

We note that if $\Psi= \{ \psi_{\ell} ~:~ \ell=1, \dots N\}$, a Parseval wavelet frame for $ L^2(\IR^n)$,
is constructed under the conditions in Theorem 1, then $\Psi$ is an $A$-FMRA based Parseval wavelet frame. First, we have that the function $\varphi$ defined by \eqref{phi111111} is refinable and $\widehat{\varphi}$ is $A^*$-locally nonzero at the origin. Thus according to  Lemma F and Theorem H,
 $ \{ U_j: = \overline{\textrm{span}}\{ d_A^{j/2}  \varphi (A^j \cdot- {\k}) ~:~ {\k} \in \IZ^n \} \}_{j \in \IZ} $
is an $A$-FMRA.
Finally, by the proof of Theorem 4 and Lemma~C,
$\Psi \subset U_1$ follows. This proves our assertion.

We also note that  if $\phi$ is a refinable functions involved in Theorem \ref{thm:main}, observe that the origin must be an $A^*$-density point for the support of $\widehat{\phi}$.
Otherwise the origin is not able to be a point of $A^*$-approximate continuity for $S|\widehat{\phi}|^2$ when we set  $S({\bf 0})|\widehat{\phi}({\bf 0})|^2=1$.\\

\subsection{Proof of Theorem \ref{thm:main2} and Corollary \ref{cor:main2}.}

Using Theorem \ref{thm:main} and results on multiresolution analyses we prove Theorem \ref{thm:main2}.

\begin{proof}[Proof of Theorem \ref{thm:main2}]
We see that \mbox{i)} implies \mbox{ii)}. According to Lemma D and Lemma~E, there exist $H_0, H_1,\dots, H_N$ some
$\IZ^n$-periodic measurable functions such that
$$
\widehat{\varphi}(A^*{\t})=H_0 ({\t}) \widehat{\varphi}({\t}) \quad \textrm{a.e., and} \quad
\widehat{\psi_{\ell}}(A^*{\t})=H_{\ell} ({\t}) \widehat{\varphi}({\t}), \quad {\ell =1,\dots,N} \quad \textrm{a.e.}
$$
because  $\{ \psi_1,\dots, \psi_N\} \subset V_1$. Since  $\{ \varphi( \cdot - {\k}) : {\k} \in \IZ^n\}$
is a Parseval frame for $V_0$, $H_0$ can be taken  in $L^{\infty}(\IT^n) $ by Lemma F.
Indeed, it can be assumed that $H_{\ell} (\t)=0$, $\ell=0,1, \dots, N$, a.e. on $\mathcal{N}_{\widehat{\varphi}} $.

Let $\Theta$ be the fundamental function associated to $H_0, H_1,\dots, H_N$  defined as in \eqref{fundafun}.
 Since $\Psi$ is a Parseval frame for $L^2(\IR^n)$, we have already seen that $\Theta$ satisfies the
 condition $\mbox{ii)}$ in Theorem \ref{thm:main} if we consider $\Theta$ instead of $S$. It remains to see that $\Theta$,   $H_1,\dots, H_N$ are in $L^{\infty}(\IT^n)$, and
 the origin is a point of $A^*$-approximate continuity for $\Theta$ is we set $\widehat{\Theta}({\bf 0})=1$.

 We have seen that  $\Theta({\t})
 \sum_{{\k} \in \IZ^n}  | \widehat{\phi}({\bf t}-{\k})|^2 \in L^{\infty}(\IT^n)$ in the proof of Theorem \ref{thm:main}.
  Furthermore, since  $\sum_{{\k} \in  \IZ^n} |\widehat{\varphi} ({\t} - {\k})|^2=1$ a.e. $\t \in \IT^n \setminus
  \mathcal{N}_{\widehat{\varphi}} $ and   bearing in mind that $\Theta (\t)=0$ a.e. on $\mathcal{N}_{\widehat{\varphi}} $,
  we conclude that  $\Theta$ is in $L^{\infty}(\IT^n)$. Hence  $H_1,\dots, H_N$ are in $L^{\infty}(\IT^n)$.

By $\mbox{ii)}$ in Theorem \ref{thm:main} we know that the origin is a point of $A^*$-approximate continuity for
 $\Theta |\widehat{\varphi}|^2$ is we set $\Theta({\bf 0})|\widehat{\varphi}({\bf 0})|^2=1$.  By Theorem H,
 the origin is a point of $A^*$-approximate continuity for $\widehat{\varphi}$ is we set $\widehat{\varphi}({\bf 0})=1$.
 Hence, the origin is a point of $A^*$-approximate continuity for $\Theta$ is we set $\Theta({\bf 0})=1$ follows.

We check that $ii)$ implies $i)$. By Lemma D, $\Psi \subset V_1$. The origin is a point of $A^*$-approximate continuity
for $S |\widehat{\varphi}|^2$, if we set $S({\bf 0})|\widehat{\varphi}({\bf 0})|^2=1$, because  the origin is a point
of $A^*$-approximate continuity for $S$ and $ |\widehat{\varphi}|^2$, if we set $S({\bf 0})=1$ and
 $|\widehat{\varphi}({\bf 0})|^2=1$ respectively.
 We conclude that $\Psi$ is a Parseval wavelet frame for $L^2(\IR^n)$ by Theorem \ref{thm:main}. Hence the proof is finished.
\end{proof}

To prove Corollary \ref{cor:main2}, we need the following condition of orthogonality.
\begin{lem} \label{lem:ort}
Let  $\varphi\in L^2(\IR^n)$ such that $\{\tau_{\k}\phi ~:~ {\k}\in
\IZ^{n}\}$ is a Parseval frame for $V$, a subspace of $L^2(\IR^n)$.
 Assume that there exists $H_0 \in L^{\infty}(\IT^n)$ such that
\begin{equation} \label{H0orto}
\widehat{\varphi}(A^*{\t}) = H_0({\t}) \widehat{\varphi}({\t}), \qquad \textrm{a.e. on } \IR^n.
\end{equation}
Let $f \in D_A(V)$ such that
\begin{equation} \label{Hforto}
\widehat{f}(A^*{\t}) = H_f({\t}) \widehat{\varphi}({\t}), \qquad \textrm{a.e. on } \IR^n,
\end{equation}
for some function $H_f$ in $L^2(\IT^n)$. Then $f$
is orthogonal to the subspace $V$ if and only if
\begin{equation} \label{sumHfH0ort}
  \sum_{k =0}^{d_A-1} H_f({\t}+  \mathbf{p}_k ) \overline{H_{0}({\t} +  \mathbf{p}_k )}  = 0, \quad \textrm{for a.e.} \,  {\t} \in \IR^n \setminus \mathcal{N}_{\widehat{\varphi}}.
 \end{equation}
\end{lem}
\begin{proof}
Let $ {\k} \in \IZ^n$. We compute the following. By Parseval's equality,
\begin{eqnarray*}
  \langle f, \tau_{\k}\varphi \rangle  &= & \langle\widehat{f}, \widehat{\tau_{\k} \varphi}  \rangle = \int_{\IR^n} \widehat{f}(\t) \overline{\widehat{\phi}({\t})} e^{2 \pi i {\k} \cdot {\bf t}} \, d{\bf t} \\
    &= & \sum_{{\bf m} \in \IZ^n}\int_{[0,1]^n - {\bf m}} D_{A^*}\widehat{f}({\bf s})  D_{A^*} \overline{\widehat{\phi}({\bf s})} ) \, e^{2 \pi i {\k} \cdot (A^*{\bf s})} \, d{\bf s}.
\end{eqnarray*}
Since $\widehat{f}$ and $\widehat{\varphi}$ are in $L^2(\IR^n)$ and $A^*(\IZ^n) \subset \IZ^n$, we have
\begin{equation*}
\langle f, \tau_{\k}\varphi \rangle =   \int_{[0,1]^n } \sum_{{\bf m} \in \IZ^n} \tau_{{\bf m}} \left( D_{A^*} \widehat{f}({\bf t})  D_{A^*} \overline{\widehat{\phi}({\bf t})} \right) e^{2 \pi i {\k} \cdot (A^*{\bf t})} \, d{\bf t}.
\end{equation*}
By \eqref{H0orto} and \eqref{Hforto},
\begin{eqnarray*}
\langle f, \tau_{\k}\varphi \rangle  &=&  d_A \int_{[0,1]^n } H_f({\t})  \overline{H_0({\t})} [\widehat{\varphi}, \widehat{\varphi}]({\bf t})  e^{2 \pi i {\k} \cdot (A^*{\bf t})} \, d{\bf t} \\
&=&  d_A \int_{[0,1]^n } H_f({\t})  \overline{H_0({\t})} \chi_{\IR^n \setminus  \mathcal{N}_{\widehat{\varphi}}}({\bf t})  e^{2 \pi i {\k} \cdot (A^*{\bf t})} \, d{\bf t},
\end{eqnarray*}
where the last equality holds according to Theorem C.
Bearing in mind that  $H_0$, $\chi_{\IR^n \setminus  \mathcal{N}_{\widehat{\varphi}}}$,   and  $e^{2 \pi i {\k} \cdot (A^*\cdot)} $ are bounded and $\IZ^n$-periodic functions and $H_f \in L^2(\IT^n)$, (ii) in Lemma~K implies that
\begin{eqnarray} \label{ortoproduct}
 \langle f, \tau_{\k}\varphi \rangle  &= &  \int_{[0,1]^n }  \left( \sum_{k =0}^{d_A-1} H_f((A^*)^{-1}{\t}+  \mathbf{p}_k )
 \overline{H_{0}((A^*)^{-1} {\t} +  \mathbf{p}_k )} \right) \\
 & & \ \ \ \ \ \times \  \chi_{\IR^n \setminus  \mathcal{N}_{\widehat{\varphi}}} ((A^*)^{-1}(\t))e^{2 \pi i {\k} \cdot {\bf t}} \, d{\bf t}. \nonumber
\end{eqnarray}
We see the necessity condition. By hypotheses we know that $f$ is orthogonal to $\tau_{\k}\varphi$, $\forall {\k} \in \IZ^n$. Thus,   \eqref{ortoproduct} gives that
$$
 \left( \sum_{k =0}^{d_A-1} H_f((A^*)^{-1}{\t}+  \mathbf{p}_k ) \overline{H_{0}((A^*)^{-1} {\t} +  \mathbf{p}_k )} \right) \chi_{\IR^n \setminus  \mathcal{N}_{\widehat{\varphi}}} ((A^*)^{-1}(\t)) = 0 \quad \textrm{a.e.},
$$
 which is the condition \eqref{sumHfH0ort}.

To see the sufficient condition, observe that $f$ is orthogonal to $V$ if an only if $f$ is orthogonal to all the generators of $V$, i.e. $\tau_{{\k}}\varphi$, $\forall {\k} \in \IZ^n.$ Therefore the proof is finished directly from  \eqref{sumHfH0ort} and \eqref{ortoproduct}.
\end{proof}

\begin{proof}[Proof of Corollary \ref{cor:main2}]
By Theorem \ref{thm:main2}, Lemma \ref{lem:ort} and the definition of $W_j$, $j \in \IZ$,  the statements follow.
\end{proof}

In these last paragraphs, we will see that the condition  at the origin  of the Fourier Transform of an involved refinable function in Extension Principles assumed by Han \cite{H:10}, \cite{H:12} is not equivalent to the condition used in this note.

The linear space of all compactly supported $C^{\infty}(\IR^n)$ (test) functions
with the usual topology will be denoted by  $\mathcal{D}(\IR^n)$. For $g \in\mathcal{D}(\IR^n) $ and $ f \in L^1_{loc}(\IR^n)$, with some abuse in the notation,  we will use
$$
\langle f,g\rangle = \int_{\IR^n} f({\t}) \overline{g({\t})} \, d{\t}.
$$

According to our context, the condition used by Han may be written as follows. Given  $f \in L^1_{loc}(\IR^n) $,  the following identity holds in the sense of distributions:
\begin{equation}\label{originH}
  \lim_{j \to \infty} |f(A^{-j} {\t})| =1,
\end{equation}
more precisely,
\begin{equation*}
  \lim_{j \to \infty} \langle |f(A^{-j} {\t})|, g\rangle = \langle 1,g\rangle, \qquad \forall g \in \mathcal{D}(\IR^n).
\end{equation*}

Let $f$ be the function in  $L^1_{loc}(\IR)$  defined by
 $$
 f(x)= \chi_{[-1,1]}(x) + \sum_{\ell=0}^{\infty} 2^{\ell + 2} \chi_{(2^{-\ell -1}, 2^{-\ell -1} + 2^{-2\ell -2})}(x),
 $$
 We will see that the origin is a point of approximate continuity of $f$
 but \eqref{originH} is not satisfied with $A=2$.

 Denote $F = [-1,1] \setminus \bigcup_{\ell=0}^{\infty} (2^{-\ell -1}, 2^{-\ell -1} + 2^{-2\ell -2})$.  Since
 \begin{eqnarray*}
 \lim_{j \to \infty} \frac{| 2^{-j}[-1,1] \bigcap F |}{| 2^{-j}[-1,1]|}
 & = & 1 -  \lim_{j \to \infty}  \frac{| 2^{-j}[-1,1] \bigcap ( \bigcup_{\ell=0}^{\infty} (2^{-\ell -1}, 2^{-\ell -1} + 2^{-2\ell -2}) )|}{| 2^{-j}[-1,1]|} \nonumber \\
 & = & 1 -  \lim_{j \to \infty}  \frac{| \bigcup_{\ell=j}^{\infty} (2^{-\ell -1}, 2^{-\ell -1} + 2^{-2\ell -2})|}{| 2^{-j}[-1,1]|} =1,
 \end{eqnarray*}
 the origin is a point of density for the set $F$.  It follows  rapidly that the origin is a point of approximate continuity of $f$.

Now we see that the function $f$ does not satisfies the condition \eqref{originH} with $A=2$.  Take $g\in \mathcal{D}(\IR)$ such that $g$ is non negative, with value $1$ on the interval $[-1,1]$, supported on $[-2,2]$, increasing on $[-2,-1]$ and decreasing on $[1,2]$.
We have the following inequalities
\begin{eqnarray*}
 \lim_{j \to \infty} \langle |f(2^{-j} {t})|, g\rangle 
 & \geq &   \lim_{j \to \infty} \int_{-1}^1 f(2^{-j}{t})  \, d{t}
 =  \lim_{j \to \infty} 2^j \int_{-2^{-j}}^{2^{-j}} f({y})  \, d{y} \\ & = & 2 +   \lim_{j \to \infty} 2^j \int_{-2^{-j}}^{2^{-j}} \sum_{\ell=0}^{\infty} 2^{\ell + 2} \chi_{(2^{-\ell -1}, 2^{-\ell -1} + 2^{-2\ell -2})}(y ) \, d{y} \\
  & = & 
  2 +    \sum_{k=0}^{\infty}  2^{-k } =4 >  \langle 1, g \rangle.
 \end{eqnarray*}
 This implies that $\eqref{originH}$ does not hold.



\end{document}